\documentclass[11pt,letterpaper]{article}
\usepackage{fullpage}
\usepackage{amsmath}
\usepackage{amsfonts}
\usepackage{amssymb}
\usepackage{amsthm}
\usepackage{tikz-cd}
\usepackage{cleveref}
\usepackage{xcolor}
\definecolor{orange}{rgb}{1,0.5,0}

\title{\bf Cubic Extremal Transition and Gromov-Witten Theory}
\author{Rongxiao Mi}
\date{\vspace{-5ex}}

\newtheorem{thm}{Theorem}[section]
\newtheorem{lemma}[thm]{Lemma}

\newtheorem{prop}[thm]{Proposition}
\newtheorem{thmx}{Theorem}
 
\newtheorem{conj}[thm]{Conjecture}

\theoremstyle{definition}

\newtheorem{definition}[thm]{Definition}
\newtheorem*{ack}{Acknowledgement}
\newtheorem*{plan}{Plan of the paper}
\newtheorem*{notation}{Notation}

\def\th@remark{%
	\thm@headfont{\bfseries}%
	\normalfont % body font
	\thm@preskip\topsep \divide\thm@preskip\tw@
	\thm@postskip\thm@preskip
}
\theoremstyle{remark}
\newtheorem{remark}[thm]{Remark}
\numberwithin{equation}{section}

\newcommand{\bb}{\mathbb}

\newcommand{\ev}{\mathrm{ev}}

\newcommand{\Li}{\mathrm{Li}}

\newcommand{\cM}{\overline{\mathcal{M}}}
\newcommand{\cL}{\mathcal{L}}
\newcommand{\cH}{\mathcal{H}}
\newcommand{\cV}{\mathcal{V}}
\newcommand{\ZZ}{\mathbb{Z}}
\newcommand{\CC}{\mathbb{C}}
\newcommand{\cO}{\mathcal{O}}
\newcommand{\PP}{\mathbb{P}}
\newcommand{\DD}{\mathbb{D}}
\newcommand{\NN}{\mathbb{N}}

\def\res{\mathop{\mathrm{Res}\:}\limits}

\begin{document}

\maketitle
\begin{abstract}
In this article, we study the change of genus zero Gromov-Witten invariants under cubic extremal transitions, following Lee-Lin-Wang \cite{MR3751817}. We use the language of quantum $D$-modules.
\end{abstract}

\tableofcontents

\section{Introduction}

The classification of Calabi-Yau 3-folds is always an important problem in algebraic geometry. After physicists discovered millions of Calabi-Yau 3-folds in the early 1990s, resolution of this problem seems to be very remote. A popular classification scheme in the early 1990s was to classify them up to well-known algebraic surgeries and was referred to as {\em Miles Reid's fantasy} \cite{MR909231}. To be more specific, a Calabi-Yau 3-fold usually admits certain algebraic surgeries from Mori's program, which are \emph{flops} and \emph{extremal transitions}. Flops are a birational morphism obtained by contracting certain rational curves and then resolving them in a different way while preserving the Calabi-Yau condition. On the other hand, extremal transitions are a birational contraction followed by a smoothing. Among extremal transitions, a well-known example is the {\em conifold transition}, namely, to locally contract a $O(-1)+O(-1)$-rational curve to a point with conifold singularity and then smooth out the singularity by deforming its defining equation. While the contraction is an important birational operation in Mori's theory, the second part of the operation, {\em a smoothing}, is {\em not} birational. One can decompose the Mori's contraction into primitive ones in which $b_2$ is reduced only by one.
We call the resulting extremal transition {\em primitive}. Furthermore, Reid's fantasy was implemented for all Calabi-Yau complete intersections in toric varieties \cite{Avram:1995pu, Chiang:1995hi} (known as a web of Calabi-Yau 3-folds).
Assuming Reid's fantasy, the hope is that one may be able to prove {\em mirror symmetry}
for \emph{all} Calabi-Yau 3-folds by establishing a mirror symmetry statement among algebraic surgeries and checking it for one mirror pair.

In the past two decades, a tremendous amount of techniques have been developed for computing Gromov-Witten invariants and proving mirror symmetry. We still have very little tools to go beyond toric cases. Therefore, surgery remains an important technique for proving mirror symmetry for general Calabi-Yau 3-folds. Li-Ruan \cite{MR1839289} pioneered this line of attack in the middle of the 1990s where they calculated the change of Gromov-Witten invariants for flops and conifold transitions. During the process, they invented an important tool called {\em degeneration formula}
to reduce the problem to a local model which is often toric and thus can be calculated explicitly. In particular, they showed that a flop between 3-folds would induce an isomorphism of quantum cohomology rings after performing appropriate analytic continuation of a quantum variable. Since then, analytic continuation of Gromov-Witten theory has received a lot of interest. Together with the subsequent work on crepant resolutions of orbifolds \cite{MR3112518}, Li-Ruan's work on flops has generated an impressive new direction of the subject with many important problems such as {\em crepant resolution conjecture}, {\em crepant transformation conjecture}, {\em LG/CY correspondence} and so on.

A less known part of Li-Ruan's pioneer work \cite{MR1839289} is on conifold transitions. %Let me be more specific. 
One considers a more general {\em small transition} where certain rational curves get contracted which are not necessarily $O(-1)+O(-1)$-type. These curves can be deformed into a collection of conifold transitions. Since Gromov-Witten invariants are deformation invariants, a general small transition can be dealt with in the same way as conifold transitions. 
The second part of Li-Ruan's work is that a small transition induces a homomorphism between their quantum cohomology rings.

Conifold transitions are a very special case of extremal transitions. According to Wilson's classification \cite{MR1150602}, there are three types of primitive contractions: one may contract a union of curves (small transition or type I), or a divisor to a point (type II), or a divisor to a curve (Type III). However, beyond small transitions, both the local model and degeneration formula become complicated. It is not clear what kind of statements can be proved.  As far as we know, there was no work on other two types of transitions during the last two decades. In this paper, we will attempt to break through this wall by studying cubic extremal transitions, which are perhaps the simplest case of Type II transitions. However, there are some other works \cite{MR3568645,2015arXiv150203277L} have been done to generalize Li-Ruan's work on small transitions in different settings.

Let $S$ be a cubic surface embedded in a smooth Calabi-Yau 3-fold $X$, where we assume that the rational curves on $S$ generate an extremal ray in the sense of Mori. Then there is a birational contraction $X\to Y_0$ in which $S$ gets contracted to a point in $Y_0$, whose local equation is given by 
$$x^3+y^3+z^3+u^3=0.$$
Then, we may smooth out the singularity by deforming the local equation
$$x^3+y^3+z^3+u^3=t \quad (t\neq 0),$$
by which we obtain another Calabi-Yau 3-fold $Y$. The process of going from $X$ to $Y$ is called a {\em cubic extremal transition}. Our goal is to study how genus zero Gromov-Witten invariants of $X$ and $Y$ are related under this transition.

Our approach begins with the study of \emph{the local model}, which is obtained by closing
up a neighborhood of $S$. We have that $X={\bf P}(K_S\oplus O)$, and $Y$ is a cubic 3-fold of ${\bf P}^4$.
The local model is interesting in its own right, as it fully captures the local surgery involved in cubic extremal transitions. It also serves as a key step when one wants to {\em globalize} the result using degeneration techniques.

%Another reason to study the local model is that we hope to use degeneration techniques to {\em globalize} the local result, which has been a tradition in studying Gromov-Witten theory under algebraic surgeries.

It was known to Li-Ruan that one could not expect a relation between numerical Gromov-Witten invariants of $X$ and $Y$ and analytic continuation is expected to play a role. Following Lee-Lin-Wang \cite{MR3751817}, we formulate our result in terms of \emph{quantum $D$-module}, whose definition can be found in Section 2. One may view the quantum $D$-module as encoding the genus zero Gromov-Witten theory (See Remark 2.6). Let $\cH(X)$ and $\cH(Y)$ be the ambient part quantum $D$-modules of $X$ and $Y$, respectively.
The main result in this paper is the following.

%Inspired by the work of Lee-Lin-Wang \cite{2017arXiv170504799L} and Irritani-Xiao \cite{MR3568645}, we formulate our result in terms of \emph{quantum $D$-module}. whose definition can be found at \cite{?}.

\begin{thmx}[=Theorem 3.7]\label{thmA}
	One may perform analytic continuation of $\cH(X)$ over the extended K\"ahler moduli to obtain a $D$-module $\bar{\cH}(X)$, then there is a divisor $E$ and a submodule $\bar{\cH}^E(X)\subseteq \bar{\cH}(X)$ with maximum trivial $E$-monodromy such that
	$$\bar{\cH}^E(X)|_E\simeq \cH(Y),$$
	where $\bar{\cH}^E(X)|_E$ is the restriction to $E$.
\end{thmx}

Unlike flops, extremal transitions usually decrease K\"ahler moduli but increase complex moduli. As a result, the Gromov-Witten theories of $X$ and $Y$ are {\em not} equivalent in general. In our study of the relation between $\cH(X)$ and $\cH(Y)$, one major challenge is that they have different ranks. \Cref{thmA} reveals the fact that the monodromy property plays a key role in relating these two $D$-modules. 

In the case above, both $X$ and $Y$ are complete intersections in toric varieties. It is well-known that Mirror Theorem identifies the (ambient part) quantum $D$-module with a GKZ-type differential equation system, whose fundamental solution is given by the components of the (ambient part) $I$-function. Here we denote the $I$-function for $X$ (resp. $Y$) by $I^X(q_1,q_2)$ (resp. $I^Y(y)$). Let $\tilde X$ (resp.  $\tilde Y$) be the ambient toric variety of $X$ (resp. $Y$), which is defined in Section 3. We may as well formulate a result in terms of their $I$-functions, which may be viewed as a numerical version of \Cref{thmA}.

\begin{thmx}[=Theorem 3.6]
	One may perform analytic continuation of $I^X(q_1,q_2)$ over the extended K\"ahler moduli to obtain $\bar I^X(x,y)$, where the change of variable is given by $x\mapsto q_1^{-1}$ and $y\mapsto q_1q_2$. Then there exists an explicit degree-preserving linear transformation $L:H^*(\tilde X)\to H^*(\tilde Y)$ such that $I^Y(y)$ is recovered by
	\[I^Y(y)=\lim_{x\to 0} L\circ \bar I^X(x,y).\]
\end{thmx}

One interesting phenomenon arising out of our study is that the rank discrepancy between $\cH(X)$ and $\cH(Y)$ may be partially explained by the Fan-Jarvis-Ruan-Witten theory (FJRW theory, for short) associated to the cubic singularity. Let $ {\cL}(W)$ be the quantum $D$-module attached to the narrow part regularized FJRW $I$-function for $(W,G)$, where $W=x^3+y^3+z^3+u^3$ and $G=\langle J_W\rangle$ (see Definition 5.1). It turns out that the rank of ${\cL}(W)$ coincides with the difference between the ranks of $\cH(X)$ and $\cH(Y)$. We have the following result.

\begin{thmx}[=Theorem 5.4]\label{thmC}
	Let $\bar \cH(X)$ be the $D$-module considered in \Cref{thmA}, then there is another divisor $F$ in the extended K\"ahler moduli and a submodule  $\bar{\cH}^F(X)\subseteq \bar{\cH}(X)$ with maximum trivial $F$-monodromy such that
	$$\bar{\cH}^F(X)/\bar{\cH}^E(X)|_F\simeq {\cL}(W).$$
\end{thmx}

Roughly speaking, \Cref{thmA} and \Cref{thmC} tell us that the genus zero Gromov-Witten theory of $X$ determines the genus zero Gromov-Witten theory of $Y$ as well as the genus zero FJRW theory for the cubic sigularity. Schematically, we have
%\begin{center}
%    \quad   GW theory($X$) $\Rightarrow$ GW theory($Y$) + FJRW theory %(singularity)\quad (for $g=0$)
%\end{center}
\[\qquad \mbox{GW theory } (X)\Longrightarrow\begin{cases} 
\mbox{GW theory } (Y) \\
\qquad\quad+\\
\mbox{FJRW theory } (W)
\end{cases}\quad(\mbox{for }g=0).
\]

A very interesting question is if we can reverse the implication above, which we will consider in our future work.

In the global case where both $X$ and $Y$ are Calabi-Yau 3-folds, it is generally impossible to write down their $I$-functions. However, they still have well-defined quantum $D$-modules. We, therefore, make the following conjecture.

\begin{conj}\label{myconj}
	The statement of \Cref{thmA} holds for any primitive extremal transition between two Calabi-Yau 3-folds.
\end{conj}

We will then verify this conjecture for one particular example, where $X$ is a Calabi-Yau 3-fold obtained by blowing up a triple point in a singular quintic, and $Y$ is a generic smooth quintic. We have the following theorem.
\begin{thmx}[=Theorem 4.4]\label{thmD}
	\Cref{myconj} holds for the cubic extremal transition going from $X$ to $Y$ as stated above.
	
\end{thmx}
The proof of the above theorem is based on the explicit descriptions of their quantum $D$-modules, as they are both hypersurfaces in toric varieties. However, such an explicit description is not available for general Calabi-Yau 3-folds. The hope is that by using degeneration techniques one may prove \Cref{myconj} directly from a local result like \Cref{thmA}.

We should mention that there is a long list of works to study the change of Gromov-Witten theory under algebraic surgeries, such as the blow-up formulas by Hu \cite{MR1759269,MR1818985}, simple and ordinary flops by Lee-Lin-Wang \cite{MR2680420, MR3568338, MR3568339}, crepant resolution by Coates-Ruan \cite{MR3112518}, Coates-Corti-Iritani-Tseng \cite{CCIT1, MR2510741} and so on. We refer the reader to these papers for a complete reference.

\begin{plan}
	In Section 2, we will fix our notations and reformulate Li-Ruan's result \cite[Corollary B.1]{MR1839289} on conifold transitions in our current setting. The local model is studied in Section 3, where we prove Theorem A and Theorem B. In Section 4, we will prove \Cref{thmD} and thus verify \Cref{myconj} in a special case. In Section 5, we describe the genus zero FJRW $D$-module of the cubic singularity and then prove \Cref{thmC}. Throughout the paper, we will only consider homology/cohomology in even degrees, and we will freely switch between the perspectives of locally free $D$-modules and local systems in the proofs of our main results.
\end{plan}
%\begin{comment}
\begin{ack}
	I am grateful to my advisor Yongbin Ruan, for introducing the problem and for all of his guidance and support. I would like to thank Yuan-Pin Lee for reading and commenting an early draft of the paper. This work is inspired by Chin-Lung Wang's talk on simple flips at the 2016 Workshop on Global Mirror Symmetry at Chern Institute of Mathematics. I would also like to thank Tyler Jarvis, Felix Janda, Ming Zhang, Rachel Webb, Robert Silversmith for helpful conversations. 
\end{ack}

\section{Quantum $D$-modules and Conifold Transitions}

In this section, we will study the quantum $D$-module aspect of the Li-Ruan's result \cite[Corollary B.1]{MR1839289} on conifold transitions. To fix our notation, we review some basics of Gromov-Witten theory. General references on this subject are \cite{MR1677117,MR1492534,MR2003030}.
\subsection{Gromov-Witten invariants and quantum $D$-modules} 

Let $P$ be a smooth projective variety, $H^*(P)$ be its cohomology ring in even degrees, with coefficients in $\bb Q$ unless specified otherwise. Let ${\cM}_{0,n}(P,\beta)$ be the moduli space of genus zero stable maps $f:(C,x_1,\cdots,x_n)\to P$ from a rational nodal curve $C$ with $n$ markings and $f_*[C]=\beta\in NE(P)$, the Mori cone of effective curves. Given $\gamma_1, \cdots, \gamma_n\in H^*(P)$, $a_1,\cdots, a_n\in \NN$, the descendent Gromov-Witten invariant is defined as
\[\langle \tau_{a_1}(\gamma_1)\cdots \tau_{a_n}(\gamma_n)\rangle_{\beta}:=\int_{[\cM_{0,n}(P,\beta)]^{vir}}\psi_1^{a_1}\ev_1^*\gamma_1\cdots \psi_n^{a_n}\ev_n^*\gamma_n,\]
where
\begin{itemize}
	\item $\ev_i: \cM_{0,n}(P,\beta)\to P$ is the $i$-th evaluation map.
	\item $\psi_i=c_1(\bb L_i)$, where $\bb L_i$ is the line bundle over $\cM_{0,n}(P,\beta)$ whose fiber over a moduli point $[(C,x_1,\cdots, x_n, f)]$ is the cotangent line $T^*_{x_i}C$ at the $i$-th marked point.
\end{itemize}
When $a_1=\cdots=a_n=0$, we call this primary Gromov-Witten invariant, and write it as
\[\langle \gamma_1\cdots \gamma_n\rangle_{\beta}:=\int_{[\cM_{0,n}(P,\beta)]^{vir}}\ev_1^*\gamma_1\cdots \ev_n^*\gamma_n.\]

To introduce the notion of the quantum $D$-module of $\cH(P)$, we choose a basis $\{T_i\}_{i=0}^m$ for $H^*(P)$ such that $T_0=\mathbf{1}$ and $\{T_1,\cdots,T_r\}$ is a nef basis for $H^2(P)$. Let $\{T^i\}$ be the dual basis for $\{T_i\}$. Choose a generic cohomology class $T=\sum t_iT_i$ with coordinate $t=(t_i)_{i=0}^m$, then the genus zero
Gromov-Witten potential is defined as
\[\Phi(T):=\sum_{n=0}^\infty\sum_{\beta\in NE(P)}\frac{Q^\beta}{n!}\langle \underbrace{T,T,\cdots,T}_{\text{\normalsize$n$-tupe}}\rangle_\beta.\]
The (big) quantum product is defined as
\[T_i\star_{t} T_j:=\sum_{k=0}^m \Phi_{ijk}T^k,\]
where the structure coefficients are $\Phi_{ijk}=\partial_{t_i}\partial_{t_j}\partial_{t_k} \Phi$. This product is associate due to WDVV equation.

In Givental's approach to the mirror theorem \cite{MR1653024}, the following generating function is introduced, which is usually called the {\em big $J$-function}
\[J_{\mathrm{big}}(t,z^{-1}):=1+\frac{T}{z}+\sum_{n=0}^{\infty}\sum_{\beta\in NE(P)}\frac{Q^\beta}{n!}\sum_{k=0}^m\left\langle \frac{T_k}{z(z-\psi_1)}, \underbrace{T,T,\cdots,T}_{\text{\normalsize$n$-tupe}}\right\rangle_\beta T^k.\]
%By topological recursion relation, $J_{\mathrm{big}}(t,z^{-1})$ gives a fundamental solution to the following {\em quantum differential equation}
%\[(z\partial_{t_i})J=T_i\star_{t}J,\qquad\forall i=0,1,\cdots,m.\]

It is usually convenient to consider the small $J$-function, which is obtained by restricting the big $J$-function to $t_{r+1}=\cdots=t_m=0$. Let $q_i=e^{t_i}$ for $i=1,\cdots,r$. These variables are usually called {\em small parameters} and may be viewed as local coordinates of the K\"ahler moduli space. In this setting, the Novikov varible $Q$ may be eliminated by setting $Q\equiv 1$. By divisor equation, the small $J$-function takes the following form
\begin{equation}\label{smallj}
J_\mathrm{sm}(t_0,q,z^{-1}):=e^{t_0/z}q^{T/z}\left(1+\sum_{\beta\neq 0}q^\beta\sum_{k=0}^m\left\langle \frac{T_k}{z(z-\psi_1)}\right\rangle_\beta T^k\right),
\end{equation}
where $$q^{T/z}:=\prod_{i=1}^{r}q_i^{T_i/z},\qquad q_i^{T_i/z}:=\exp\left(\frac{T_i}{z}\log q_i\right),\qquad q^\beta:=q_1^{T_1\cdot\beta}\cdots q_r^{T_r\cdot\beta}.$$
%=\exp\left(\frac{t_0+\sum_{i=1}^r T_i\log q_i}{z}\right)

Let $\DD_q$ be the ring generated by the log differential operator $z\delta_{q_i}:=zq_i\partial_{q_i}$ over $\CC[q_i,q_i^{-1},z]$, where we view $z$ as a formal parameter. Let $J(q,z^{-1}):=J_\mathrm{sm}(0,q,z^{-1})$. We make the following definition
\begin{definition}
	The (small) {\em quantum $D$-module} $\cH(P)$ is the cylic $\DD_q$-module generated by $J(q,z^{-1})$. 
\end{definition}
\begin{remark}
	It is well known that the cyclic $D$-module generated by $J_\mathrm{big}(t,z^{-1})$ can be identified with the Dubrovin connection associated to the big quantum product. Our definition above coincides with the Dubrovin connection restricted to the small parameter space $H^2(P)$.
\end{remark}

\subsection{Quantum Lefschetz theorem and ambient quantum $D$-modules}

Having defined quantum $D$-module for a general smooth projective variety $P$, now we turn into the case where $Z$ is a smooth complete intersection inside $P$. Let $L_i$ be convex line bundles over $P$ and $\cV=\oplus L_i$, and assume $Z=\sigma^{-1}(0)$ for a section $\sigma\in H^0(\oplus \cV)$. The quantum Lefschetz theorem gives a way to compute $QH^*(Z)$ from $QH^*(P)$.

In Coates-Givental's work \cite{MR2276766}, a twisted version of $J$-function is introduced for the pair $(P,\cV)$ as follows.
\[J_\mathrm{big}^\cV(t,z^{-1}):=1+\frac{T}{z}+\sum_{n=0}^\infty\frac{Q^d}{n!}(\ev_{n+1})_*\left(\frac{\cV'_{0,n+1,d}}{z(z-\psi_{n+1})}\prod_{i=1}^n\ev^*_i T\right),\]
where $\cV'_{0,n+1,d}$ is the kernel of the map $R^0\varphi_*\ev^*_{n+1}\cV\to \ev_{n+1}^*\cV$ and $\pi:\cM_{0,n+1}(P,\beta)\to \cM_{0,n}(P,\beta)$ forgets the last marking.
Let $J_\mathrm{big}^P(t,z^{-1})$ be the big $J$-function for $P$, we write
\[J_\mathrm{big}^P(t,z^{-1})=\sum_{\beta}Q^\beta J_\beta(t,z).\]
Then we define the function $I_\mathrm{big}^\cV(t,z)$ as the following
\[I_\mathrm{big}^\cV(t,z,z^{-1}):=\sum_{\beta}Q^\beta J_\beta(t,z)\prod_i\prod_{k=1}^{L_i\cdot\beta}(c_1(\cL)+kz).\] 

When $c_1(Z)\geq 0$, the mirror theorem shows that $J_\mathrm{big}^P(\tau(t),z^{-1})=I_\mathrm{big}^P(t,z^{-1})$ for the mirror transformation $t\mapsto \tau(t)$. However, without the restriction $c_1(Z)\geq 0$, the function $I_\mathrm{big}^\cV(t,z,z^{-1})$ may have positive powers of $z$. In this situation, a fundamental result in Coates-Givental's work \cite{MR2276766} is the following.
\begin{thm}\label{mirror}
	$J_\mathrm{big}^\cV(\tau,z)$ is recovered from $I_\mathrm{big}^\cV(t,z,z^{-1})$ via the Birkhoff factorization procedure followed by a (generalized) mirror transformation $t\mapsto \tau(t)$.
\end{thm}
Let $J_\mathrm{big}^Z(t,z^{-1})$ be the big $J$-function for $Z$. The full genus zero Gromov-Witten theory of $Z$ is hard to determine due to the existence of primitive cohomology classes, which lie in the kernel of $\iota_*$. However, we have the following relation
\[\mathbf{e}(\cV)J_\mathrm{big}^\cV(t,z^{-1})=\iota_*J_\mathrm{big}^Z(\iota^*T,z^{-1}).\]
As before, we consider the small version of the big $J$(resp. $I$)-functions. By setting $t_{r+1}=\cdots=t_m=0$, $Q\equiv 1$ and $q:=(q_i)=(e^{t_i})$ for $i=1,\cdots,r$, we obtain the (multivalued) small $J$-function $J_\mathrm{sm}^\cV(t_0,q,z^{-1})$ (resp. small $I$-function $I_\mathrm{sm}^\cV(t_0,q,z,z^{-1})$). Now we make the following definition for the ambient part quantum $D$-module of $Z$.
\begin{definition}
	The {\em ambient part quantum $D$-module} of $Z$ is the cyclic $\DD_q$-module generated by $\mathbf{e}(\cV)J_\mathrm{sm}^\cV(0,q,z^{-1})$, which is still denoted by $\cH(Z)$ when there is no risk of confusion.
\end{definition}
\begin{remark}
	It follows directly from \Cref{mirror} that the ambient part quantum $D$-module of $Z$ may be identified with the cyclic $\DD$-module generated by $\mathbf{e}(\cV)I_\mathrm{sm}^\cV(0,q,z,z^{-1})$. It is often easy to write down the explicit expression of $\mathbf{e}(\cV)I^\cV(0,q,z,z^{-1})$ provided that $P$ is a toric projective variety or certain GIT quotients. The left ideal in $\DD_q$ that annihilates $I_\mathrm{sm}^\cV(0,q,z,z^{-1})$ is called {\em Picard-Fuchs ideal}.
\end{remark}
\begin{remark}
	Suppose $H^*(P)$ is generated by divisors, then by a reconstruction theorem \cite{MR2102400}, we can recover the genus zero Gromov-Witten potential of $P$ from the small $J$-function. In the case above where $Z$ is a complete intersection, this reconstruction procedure yields all the genus zero Gromov-Witten invariants in which the insertions are pulled back from the ambient space $P$. This allows us to use the ambient part quantum $D$-module $\cH(Z)$ to represent the genus zero Gromov-Witten theory of $Z$. 
\end{remark}
\begin{remark}
	One may also adopt the viewpoint of integrable connections to define the ambient part quantum $D$-module. A closely related definition for toric nef complete intersections is given by Mann-Mignon \cite{MR3663796}. 
\end{remark}

\subsection{Conifold transitions and Li-Ruan's theorem}
In this section, we will reformulate Li-Ruan's result on conifold transitions \cite[Corollary B.1]{MR1839289}. Let $P$ be a smooth Calabi-Yau 3-fold, then the genus zero Gromov-Witten theory of $P$ boils down to the following numbers
\[N_\beta^P:=\deg([\cM_{0,0}(X,\beta)]^{vir})\in \bb Q\qquad \forall \beta\in NE(X),\]
as the virtual dimension of $\cM_{0,0}(X,\beta)$ is zero. 

Let $X$ and $Y$ be smooth Calabi-Yau 3-folds such that $Y$ is obtained by contracting finitely many $O_{\PP^1}(-1)+O_{\PP^1}(-1)$-rational curves on $X$ followed by a smoothing. This surgery is called a {\em conifold transition}. The following theorem is proved in \cite{MR1839289}.
\begin{thm}[Li-Ruan]\label{lir} The conifold transition from $X$ to $Y$ induces the following morphisms:
	\[\varphi_*:H_2(X)\to H_2(Y), \quad \varphi^*:H^*(Y)\to H^*(X)\]
	such that
	\begin{enumerate}
		\item $\varphi_*$ is surjective;
		\item $\varphi^*:H^2(Y)\to H^2(X)$ is dual to $\varphi_*$ and the dual map to $\varphi_*:H^4(Y)\to H^4(X)$ gives a right inverse to $\varphi_*$;
		\item For every $0\neq\beta'\in H_2(Y)$, the set $\Lambda=\{\beta\in NE(X):\varphi_*(\beta)=\beta'\}$ is finite;
		\item For every $0\neq\beta'\in H_2(Y)$, the following relation holds 
		\[\sum_{\beta\in \Lambda}N^X_{\beta}=N^Y_{\beta'}.\]
	\end{enumerate}
\end{thm}
We will study the implication of this theorem on quantum $D$-modules. For a Calabi-Yau 3-fold $P$, the small $J$-function takes a simpler form, which is proved in \cite[Section 10.3.2]{MR1677117}.
\begin{lemma}\label{cyJ}
	Setting $t_0=0$, the J-function of $P$ defined in \Cref{smallj} is
	\[J^P(q,z)=q^{T/z}\left(1+z^{-2}\sum_{\beta\neq 0} q^\beta N_\beta^P[\beta] -2z^{-3}\sum_{\beta\neq 0} q^\beta N_\beta^P[pt]\right),\]
	where $[\cdots]$ means the Poincare dual.
\end{lemma}
This lemma is useful way in comparing the small $J$-functions of $X$ and $Y$. The topological change of cornifold transitions has been studied in \cite{MR3568645}. It is shown that there is an exact sequence:
\begin{center}
	\begin{tikzcd}
	\oplus_{i=1}^k\bb C[E_i] \arrow[r] & H_2(X)\arrow[r,"\varphi_*"] & H_2(Y)\arrow[r] &0,
	\end{tikzcd}
\end{center}
where $\varphi_*$ is the morphism in \Cref{lir}, and $[E_i]$ are the exceptional curve classes in the contraction of $X$. According the multiple cover formula, we have
\[N^X_{nE_i}=\frac{1}{n^3},\quad i=1,2,\cdots,k.\]

Following \cite{MR3568645}, we choose a basis $b_1,\cdots,b_{r+m}$ of $H_2(X)$ in such a way that $b_1,\cdots,b_{r+m}$ span a cone containing the Mori cone of $X$, and the last $m$ elements $b_{r+1},\cdots, b_{r+m}$ span a cone containing the classes $[E_1],\cdots,[E_k]$, where $m$ is the dimension of the cone spanned by the classes $[E_1],\cdots,[E_k]$. In general we have $m\leq k$, since the curve classes $[E_i]$ have a nontrivial linear relation. Once such a basis is chosen, we get a natural basis $\varphi_*b_{1},\cdots,\varphi_*b_r$ for $H_2(Y)$. So we may use $q_1,\cdots, q_{r+m}$ as the small parameters for $X$ which are dual to $b_1,\cdots, b_{r+m}$, and $\tilde q_{1},\cdots,\tilde q_{r}$ for $Y$ dual to $\varphi_*b_{1},\cdots,\varphi_*b_{r}$. 
%As the $J$-function can be viewed as a (multivalued) power series defined over the K\"ahler moduli space (with coordinates $q_i$), we introduce the following notion    
\begin{lemma}\label{limJ}
	The small $J$-function of $X$ can be written as a sum
	\begin{equation}\label{decomp}
	J^X(q,z^{-1})=J_1^X(q,z^{-1})+J_2^X(q,z^{-1}),
	\end{equation} 
	such that
	\begin{equation}\label{rest}
	\lim_{\substack{q_i\to 1\\r+1\leq i\leq r+m}}L\circ J_1^X(q,z^{-1})=J^Y(\tilde q,z^{-1}),
	\end{equation}
	where $L:H^*(X)\to H^*(Y)$ is the dual morphism to $\varphi_*$ given in \Cref{lir}, and $\tilde q_i\mapsto q_i$ for $i=1,2,\cdots,r$.
\end{lemma}
\begin{proof}
	We define the following series:
	\[J_1^X(q,z^{-1}):=\prod_{i=1}^{b+m}q_i^{[b_i]/z}\left(1+z^{-2}\sum_{\varphi_*(\beta)\neq 0} q^\beta N_\beta^X[\beta] -2z^{-3}\sum_{\varphi_*(\beta)\neq 0} q^\beta N_\beta^X[pt]_X\right),\]
	\begin{align*}
	J_2^X(q,z) & =\prod_{i=1}^{b+m}q_i^{[b_i]/z}\left(z^{-2}\sum_{\substack{\varphi_*(\beta)= 0\\\beta\neq 0}} q^\beta N_\beta^X[\beta] -2z^{-3}\sum_{\substack{\varphi_*(\beta)= 0\\\beta\neq 0}} q^\beta N_\beta^P[pt]_X\right)\\
	&=\prod_{i=1}^{b+m}q_i^{[b_i]/z}\left(z^{-2}\sum_{i=1}^{k}\sum_{n=1}^\infty q^{nE_i}N^X_{nE_i}[nE_i] -2z^{-3}\sum_{i=1}^{k}\sum_{n=1}^\infty q^{nE_i}N^X_{nE_i}[pt]_X\right)\\
	&=\prod_{i=1}^{b+m}q_i^{[b_i]/z}\left(z^{-2}\sum_{i=1}^{k}\sum_{n=1}^\infty \frac{q^{nE_i}}{n^2}[E_i] -2z^{-3}\sum_{i=1}^{k}\sum_{n=1}^\infty \frac{q^{nE_i}}{n^3}[pt]_X\right).
	\end{align*}
	Obviously, we have \[J^X(q,z^{-1})=J_1^X(q,z^{-1})+J_2^X(q,z^{-1}).\]
	The small $J$-function of $Y$ is the following
	\[J_1^X(\tilde q,z^{-1}):=\prod_{i=1}^{b}\tilde q_i^{[b_i]/z}\left(1+z^{-2}\sum_{\beta'\neq 0}  \tilde q^{\beta'} N_{\beta'}^Y[\beta'] -2z^{-3}\sum_{\beta'\neq 0} \tilde q^{\beta'} N_{\beta'}^Y[pt]_Y\right).\]
	Since $L$ is dual to $\varphi^*$, by \Cref{lir}(1-2) we have that
	\[L[\beta]=[\varphi_*(\beta)],\quad L[pt]_X=[pt]_Y.\]
	Comibing this with \Cref{lir} gives the desired result.
\end{proof}
Let $E$ be the locus $\{q_{r+1}=\cdots=q_{r+m}=1\}$ in $\CC^{r+m}$. We call this the {\em transition locus}. Consider a small punctured neighborhood $U$ of the point $\mathbf{q}=(q_1=\cdots=q_r=0,q_{r+1}=\cdots=q_{r+m}=1)$ with each line $q_i=0(\mbox{or }1)$ deleted. Fixing $q_1,\cdots,q_r$, we choose a path to analytically continue $J^X(q,z^{-1})$ to a point in $U$, so that we obtain a function $\bar{J}^X(q,z^{-1})$. As in the decomposition (\ref{decomp}), we have
\[\bar J^X(q,z^{-1})=\bar J_1^X(q,z^{-1})+\bar J_2^X(q,z^{-1}).\]
\begin{lemma}\label{trivialm}
	$\bar J_1^X(q,z^{-1})$ has trivial $E$-monodromy, i.e. $\bar J_1^X(q,z^{-1})$ remains unchanged under analytical continuation of $\bar{J}_1^X(q,z^{-1})$ along any loop circulating around $E$ (with fixed $q_1,\cdots,q_r$).
\end{lemma}
\begin{proof}
	This follows from the fact that the set $\Lambda=\{\beta\in NE(X):\varphi_*(\beta)=\beta'\}$ is finite (in \Cref{lir}). For fixed $q_1,\cdots,q_r$, both terms 
	$\sum_{\varphi_*(\beta)\neq 0} q^\beta N_\beta^P[\beta]$ and  $\sum_{\varphi_*(\beta)\neq 0} q^\beta N_\beta^P[pt]$ are polynomials in $q_{r+1},\cdots,q_{r+m}$, thus admitting trivial $E$-monodromy after performing the analytic continuation.
	
\end{proof}
\begin{remark}\label{maxm}
	We note that $\bar J_2^X(q,z^{-1})$ has nontrivial monodromy. This follows from the fact that the polylogarithm function
	\[\Li_s(q)=\sum_{n=1}^\infty \frac{q^n}{n^s},\qquad s=2,3,\]
	is analytic in $|q|<1$ and branched at $q=1$ with the monodromy operator $M_1$ around $q=1$ given by
	\[{M}_1(\Li_s(q))=\Li_s(q)+\frac{2\pi i}{\Gamma(s)}\log^{s-1}(q).\]
	This has been computed in \cite{brown2013iterated}. 
\end{remark}
Now we are in a position to reformulate Li-Ruan's result (\Cref{lir}) as follows.
\begin{thm}
	For the conifold transition of Calabi-Yau 3-folds $X$ to $Y$, one may analytically continue the quantum $D$-module ${\cH}(X)$ to obtain a $D$-module $\bar {\cH}(X)$ over $U$, let $E$ be the transition locus $q_{r+1}=\cdots=q_{r+m}=1$, then there is a submodule $\bar{\cH}^E(X)\subseteq \bar {\cH}(X)$ which has maximal trivial E-monodromy such that
	\[{\bar{\cH}^E(X)|_E \simeq \cH}(Y),\]
	Here $\bar{\cH}^E(X)|_E$ is the restriction of $\bar{\cH}^E(X)$ to the transition locus $E$. 
\end{thm}
\begin{proof}
	By definition, ${\cH}(X)$ is identified with the $\DD_q$-module generated by $J^X(q,z^{-1})$. After performing the analytic continuation, we see that $\bar J^X_1(q,z^{-1})$ has trivial monodromy around $E$. Let $\bar{\cH}^E(X)$ be the sub-local system attached to $\bar J^X_1(q,z^{-1})$. After performing the analytic continuation, by \Cref{trivialm} and \Cref{maxm}, we see that $\bar{\cH}^E(X)$ is the submodule of $\bar {\cH}(X)$ with maximum trivial $E$-monodromy. Restriction this to $E$ amounts to taking the limit $q_i\to 1$ ($i=r+1,\cdots,r+m$). It follows from \Cref{limJ} that ${\bar{\cH}^E(X)|_E \simeq \cH}(Y)$.
\end{proof}

\section{The local model}

In this section, we will study the local model of cubic extremal transitions. Let $V$ be a Calabi-Yau 3-fold that contains a smooth cubic surface $S$. Assume that the rational curves on $S$ generate an extremal ray of the Mori cone of $V$, then we can birationally contract $S$ to a point to obtain a singular Calabi-Yau 3-fold $\bar{V}$ with local singularity given by 
\[x^3+y^3+z^3+u^3=0.\]
To smooth out the singularity, we consider the following deformation of the above equation
\[x^3+y^3+z^3+u^3=t\quad (t\neq 0).\]
This is a local surgery. According to Gross' work \cite{MR1464900}, this surgery can be done globally. In other words, there is a Calabi-Yau 3-fold $\widetilde{V}$ obtained by this smoothing. The process of going from $V$ to $\tilde{V}$ is called a {\em cubic extremal transition}.

If we only consider a small neighborhood where this surgery takes place, we obtain the {\em local model}. Viewing $S$ as a smooth cubic surface in $\PP^3$, we close up a neighborhood of $S$ and then define
\[X:=\PP(N_{S/V}\oplus O)=\PP(K_S\oplus O).\]
The second equality follows from $K_V\simeq \cO_V$ and the adjunction formula.

Smoothing out the cubic singularity gives a cubic 3-fold in $\PP^4$, i.e.
\[Y=\{x^3+y^3+z^3+u^3=tv^3\}\subseteq \bb P^4.\]
We note first that both $X$ and $Y$ can be naturally embedded into smooth projective toric varieties. Indeed, for $X$, by adjunction formula
\[K_S=K_{\bb P^3}\otimes O_{\bb P^3}(S)|_S=O_{\bb P^3}(-1)|_S,\]
so we have the following diagram
\begin{center}
	\begin{tikzcd}
	X\arrow{d}\arrow[hookrightarrow]{r}{i_X} & \bb P(O_{\bb P^3}(-1)\oplus O_{\bb P^3})\arrow[d,"\pi"] \\
	S\arrow[hookrightarrow]{r} & \bb P^3
	\end{tikzcd}
\end{center}
and $X$ is the vanishing locus of a section of $\pi^*O_{\bb P^3}(3)$. Moreover, $Y$ is a hypersurface defined by a section of $O_{\bb P^4}(3)$
\begin{center}
	\begin{tikzcd}
	& O_{\bb P^4}(3) \arrow[d] &\\
	Y\arrow[hookrightarrow, r,"i_Y"] & \bb P^4 
	\end{tikzcd}
\end{center}
\subsection{The quantum $D$-modules}
Let $\cH(X)$ (resp. $\cH(Y)$) be the ambient part quantum $D$-module of $X$ (resp. $Y$). To give an explicit description of the quantum $D$-modules, we introduce the following notations:
\begin{notation} $\cO_{\widetilde{X}}(1)$ is the anti-tautological line bundle over $\widetilde{X}=\PP(\cO_{\PP^3}\oplus\cO_{\PP^3}(-1))$. 
	\begin{itemize}
		\item $h:=c_1(\pi^*\cO_{\PP^3}(1))$, $\xi:=c_1(\cO_{\widetilde{X}}(1))$. 
		\item $p:=c_1(\cO_{\PP^4}(1))$.
	\end{itemize}
	By toric geometry, we know that $H^*(\widetilde X)$ is generated by $h,\xi$, so we use $q_1$ and $q_2$ as small parameters for $X$, which correspond to $h$ and $\xi$, respectively. On the other hand, the cohomology of $\widetilde Y:=\PP^4$ is generated by $p$, and $\dim H^2(Y)=\dim H^2(\widetilde Y)=1$, so we use $y$ as the small parameter for $Y$, which corresponds to $p$.
	We also introduce the following differential operators:
	\[z\delta_{q_1}:=zq_1\frac\partial{\partial q_1},\qquad z\delta_{q_2}:=zq_2\frac\partial{\partial q_2},\qquad z\delta_y=zy\frac\partial{\partial y}.\]
\end{notation}
The small $I$-function for $X$ is the following
\[I^X(q_1,q_2): =(3h)q_1^{h/z}q_2^{\xi/z}\sum_{(d_1,d_2)\in\NN^2} q_1^{d_1}q_2^{d_2}\frac{\prod^0\limits_{m=-\infty} (\xi-h+mz)\prod\limits_{m=1}^{3d_1}(3h+mz)}{\prod^{d_1}\limits_{m=1}(h+mz)^4\prod^{d_2}\limits_{m=1}(\xi+mz)\prod^{d_2-d_1}\limits_{m=-\infty}(\xi-h+mz)},\]
subject to the relation $h^4=h\xi-\xi^2=0$.

On the other hand, the small $I$-function for $Y$ is the following
\[I^Y(y)=(3p)y^{p/z}\sum_{d\in \NN} \frac{\prod\limits_{m=1}^{3d}(3p+mz)}{\prod\limits_{m=1}^{d}(p+mz)^5},\]
subject to the relation $p^5=0$.

According to Remark 2.5, the ambient quantum $D$-module $\cH(X)$ for $X$ may be identified with the cyclic $\DD_q$-module generated by $I^X(q_1,q_2)$, and $\cH(Y)$ may be identified with the cyclic $\DD_y$-module generated by $I^Y(y)$. Our goal is to study the relation between $\cH(X)$ and $\cH(Y)$. 

\begin{lemma}
	The Picard-Fuchs ideal associated to $I^X(q_1,q_2)$ is generated by $\triangle_1$ and $\triangle_2$, where
	\[\displaystyle \triangle_1=\left(z\delta_{q_1}\right)^3-3q_1\left(3z\delta_{q_1}+z\right)\left(3z\delta_{q_1}+2z\right)\left(z\delta_{q_2}-z\delta_{q_1}\right),\]
	\[\displaystyle \triangle_2=\left(z\delta_{q_2}\right)\left(z\delta_{q_2}-z\delta_{q_1}\right)-q_2.\] 
	In other words, the components of $I^X$ gives a full basis of solutions to $\triangle_1I=\triangle_2I=0$ at any point near the origin in $(\CC^*)^2$.
\end{lemma}
\begin{proof}
	By a ratio test, we find that $I^X(q_1,q_2)$ is holomorphic when $(q_1,q_2)\in (\CC^*)^2$ is sufficiently close to the origin. It is straightforward to check that $\triangle_1$, $\triangle_2$ annihilates $I^X(q_1,q_2)$: if we write $I^X(q_1,q_2)$ in the following way
	\[I^X(q_1,q_2)=(3h)q_1^{h/z}q_2^{\xi/z}\sum_{(d_1,d_2)\in\NN^2} q_1^{d_1}q_2^{d_2}A_{d_1,d_2},\]
	then the differential equation $ \triangle_1I= 0$ (resp. $\triangle_2I=0$) follows from the recursion relation between $A_{d_1,d_2}$ and $A_{d_1-1,d_2}$ (resp. $A_{d_1,d_2-1}$), and the cohomology relation $h^4=\xi(\xi-h)=0$ is equivalent to the fact that $A_{d_1,d_2}=0$ for $d_1<0$ or $d_2<0$. Moreover, we see that the 6 components of $I^X(q_1,q_2)$ are linearly independent due to their initial terms, but the differential equation system has at most 6-dimensional solution space by a holonomic rank computation. So we obtain a full basis to the differential equation system, and the lemma is proved.
\end{proof}

We note that $I^X$ involves two small parameters $q_1$ and $q_2$, whereas $I^Y$ involves only a single small parameter $y$. To compare them, we introduce the following auxiliary function in two variables $x$ and $y$.
\[\bar I^Y(x,y):=(3p)y^{p/z}\sum_{\substack{i\geqslant 0\\ j\geqslant 0}}x^iy^j\frac{\prod\limits_{m=-\infty}^{0}(p+mz)^4\prod\limits_{m=-\infty}^{3j-3i}(3p+mz)}{\prod\limits_{m=-\infty}^{j-i}(p+mz)^4\prod\limits_{m=1}^{j}(p+mz)\prod\limits_{m=1}^{i}(mz)\prod\limits_{m=-\infty}^{0}(3p+mz)},\]
subject to the relation $p^5=0$. This may be viewed as an extension of $I^Y(y)$ in the following sense
\begin{lemma}\label{restrict}
	$\bar I^Y(x,y)$ has trivial monodromy around $x=0$ and we have
	\[ \lim_{x\to 0} \bar I^Y(x,y)=I^Y(y).\]
\end{lemma}
\begin{proof}
	Since the initial term $y^{p/z}$ does not involve $x$, the monodromy around $x=0$ is trivial. The second part is straightforward to check.
\end{proof}    

In a similar fashion, we obtain the partial differential equation satisfied by $\bar{I^Y}(x,y)$ by studying the recursion relation in their coefficients of $x^iy^j$. However, the natural partial differential equation system attached to $\bar{I^Y}(x,y)$ has 6-dimensional solution space, but the components of $\bar{I^Y}(x,y)$ give only 4 linearly independent solutions. We introduce another two hypergeometric series as follows.

\[I_5(x,y)=x^{\frac13}\sum_{i\geqslant j\geqslant 0}x^iy^j\frac{(-1)^{i-j}\Gamma(\frac13+i-j)^4}{\Gamma(\frac 43+i)\Gamma(3i-3j+1)\Gamma(1+j)z^{2j}},\]
\[I_6(x,y)=x^{\frac 23}\sum_{i\geqslant j\geqslant 0}x^iy^j\frac{(-1)^{i-j}\Gamma(\frac23+i-j)^4}{\Gamma(\frac 53+i)\Gamma(3i-3j+2)\Gamma(1+j)z^{2j}}.\]

\begin{lemma}
	The function $\bar I^Y(x,y)$ is analytic at any point near the origin in $(\CC^*)^2$, and can be annihilated by the following differential operators:
	\[\triangle'_1:= x\left(z\delta_{y}-z\delta_{x}\right)^3-3\left(3(z\delta_{y}-z\delta_{x})+z\right)\left(3(z\delta_{y}-z\delta_{x})+2z\right)\left(z\delta_{x}\right),\]
	\[\triangle'_2:= \left(z\delta_{y}\right)\left(z\delta_{x}\right)-xy.\]
	The components of $\bar I^Y(q_1,q_2)$, together with $I_5$ and $I_6$, give a full basis of solutions to the differential equation system $ \triangle'_1I= \triangle'_2I=0$ at any point near the origin in $(\CC^*)^2$.
\end{lemma}
\begin{proof}
	The proof is parallel to that of Lemma 3.2, so we omit it.
\end{proof}
We notice that the differential equation systems $\{\triangle_1I=\triangle_2I=0\}$ and $\{\triangle'_1I=\triangle'_2I=0\}$ both have 6 dimensional solution spaces. Under an appropriate change of variable, these two systems are indeed equivalent. Our key lemma is the following:
\begin{lemma}
	The change of variable $x\mapsto q_1^{-1}$ and $y\mapsto q_1q_2$ induces an equivalence between the differential equation systems $\{\triangle_1I=\triangle_2I=0\}$ and $\{\triangle'_1I=\triangle'_2I=0\}$.
\end{lemma}
\begin{proof}
	The change of variable $x\mapsto q_1^{-1}$ and $t\mapsto q_1q_2$ yields the following relation of the differential operators:
	\[z\delta_{q_1}=z\delta_y-z\delta_x,\qquad z\delta_{q_2}=z\delta_y.\]
	It follows directly that $\triangle_2I=0$ is converted to $\triangle'_2I=0$ and vice versa. For $\triangle_1I=0$ and $\triangle'_1I=0$, one just has to notice that since $x\neq 0$, $\triangle'_1I=0$ is equivalent to
	\[\left[\left(\delta_{y}-\delta_{x}\right)^3-3x^{-1}\left(3(\delta_{y}-\delta_{x})+z\right)\left(3(\delta_{y}-\delta_{x})+2z\right)\left(\delta_{x}\right)\right]I=0,\]
	which is converted to $\triangle_1I=0$ term by term under the relation among the differential operators.
\end{proof}

\subsection{Analytic continuation of the $I$-function}
From Lemma 3.5, we see that the function $I^X(q_1,q_2)$ and $\bar I^Y(x,y)$ satisfies the same system of differential equations under an appropriate change of variable. Since they are both holomorphic on certain domains, it is expected that $\bar I^Y(x,y)$ may be obtained by analytic continuation of $I^X(q_1,q_2)$ followed by a linear transformation $L:H^*(\widetilde X)\to H^*(\widetilde Y)$. This subsection is devoted to working out this analytic continuation using Mellin-Barnes method. For a similar computation, we refer the reader to \cite{MR2672282}.

We will frequently use the following identity.
\begin{lemma}
	For any $a\in \ZZ$, we have
	\begin{equation*}
	\frac{\prod^a\limits_{m=-\infty}(u+mz)}{\prod^0\limits_{m=-\infty}(u+mz)}=\displaystyle\frac{z^a\Gamma\left(1+\displaystyle\frac uz+a\right)}{\Gamma\left(1+\displaystyle\frac uz\right)}.
	\end{equation*}
\end{lemma}
By this lemma, we can rewrite $I^X(q_1,q_2)$ and $\bar I^Y(x,y)$ in the following way.
\begin{align*}
I^X(q_1,q_2)&=3q_1^{\frac hz}q_2^{\frac \xi z}\cdot\frac{h\Gamma(1+\frac{\xi-h}z)\Gamma(1+\frac hz)^4\Gamma(1+\frac\xi z)}{\Gamma(1+\frac{3h}{z})}\cdot \\
&\quad\sum_{\substack{d_1\geqslant 0\\ d_2\geqslant 0}}q_1^{d_1}q_2^{d_2} \frac{\Gamma(1+\frac {3h}{z}+3d_1)}{\Gamma(1+\frac hz+d_1)^4\Gamma(1+\frac \xi z+d_2)\Gamma(1+\frac{\xi-h}{z}+(d_2-d_1))z^{2d_2}},\tag{3.1}
\end{align*}
subject to the relation $h^4=h\xi-\xi^2=0.$
\begin{equation*}
\bar I^Y(x,y)=\frac{(3p)y^{\frac pz}\Gamma(1+\frac pz)^5}{\Gamma(1+\frac {3p}{z})}\sum_{\substack{i\geqslant 0\\ j\geqslant 0}}x^iy^j \frac{\Gamma(1+\frac{3p}{z}+3j-3i)}{\Gamma(1+\frac pz+j-i)^4\Gamma(1+\frac pz+j)\Gamma(1+i)z^{2j}},\tag{3.2}
\end{equation*}
subject to relation $p^5=0.$
\begin{thm}\label{continue}
	One may analytically continue the function $I^X(q_1,q_2)$ to obtain a holomorphic function $\bar I^X(x,y)$ near the origin in $(\CC^*)^2$, where the change of variable is given by $x\mapsto q_1^{-1}$ and $y\mapsto q_1q_2$. There exists a degree-preserving linear transformation $L:H^*(\tilde X)\to H^*(\tilde Y)$ such that $I^Y(y)$ is recovered by
	\[I^Y(y)=\lim_{x\to 0} L\circ \bar I^X(x,y).\]
\end{thm}
\begin{proof}
	By \Cref{restrict}, we simply need to show there exists a degree-preserving linear transformation $L:H^*(\tilde X)\to H^*(\tilde Y)$ such that $\bar I^Y(x,y)=L\circ \bar I^X(q_1,q_2)$ under the change of variable $x\mapsto q_1^{-1}$ and $y\mapsto q_1q_2$.
	To begin with, for every $d_2\in \NN$, we define the following function 
	\begin{align*}
	\varphi_{d_2}(s):&=(-1)^{d_2}\displaystyle\frac{\sin(\frac{\xi-h}z)\pi}{\sin(\frac{3h}z)\pi}\cdot \frac{\sin(-\frac{3h}{z}-3s)\pi}{\sin(s-d_2-\frac{\xi-h}{z})\pi}\\
	&=(-1)^{d_2}\frac{\sin(\frac{\xi-h}z)\pi}{\sin(\frac{3h}z)\pi}\cdot \frac{\pi/\sin(s-d_2-\frac{\xi-h}{z})\pi}{\pi/\sin(-\frac{3h}{z}-3s)\pi}\\
	&=(-1)^{d_2}\frac{\sin(\frac{\xi-h}z)\pi}{\sin(\frac{3h}z)\pi}\cdot \frac{\Gamma(1+\frac{\xi-h}z+d_2-s)\Gamma(s-d_2-\frac{\xi-h}{z})}{\Gamma(1+\frac{3h}z+3s)\Gamma(-\frac{3h}z-3s)}.
	\end{align*}
	It is clear that $\varphi_{d_2}(s)$ is periodic with period 1, and takes value 1 at any integer $s\in \bb Z$.
	Using (3.1), the function $I^X(q_1,q_2)$ may be further written as:
	\begin{align*}
	I^X(q_1,q_2)&=3q_1^{h/z}q_2^{\xi/z}\frac{h\Gamma(1+\frac{\xi-h}z)\Gamma(1+\frac hz)^4\Gamma(1+\frac\xi z)}{\Gamma(1+\frac{3h}{z})}\cdot\\
	&\quad \sum_{(d_1,d_2)\in \NN^2}q_1^{d_1}q_2^{d_2} \frac{\Gamma(1+\frac {3h}{z}+3d_1)\varphi_{d_2}(d_1)}{\Gamma(1+\frac hz+d_1)^4\Gamma(1+\frac \xi z+d_2)\Gamma(1+\frac{\xi-h}{z}+(d_2-d_1))z^{2d_2}},\\
	&=3q_1^{h/z}q_2^{\xi/z}\frac{h\sin(\pi\frac{\xi-h}{z})\Gamma(1+\frac{\xi-h}z)\Gamma(1+\frac hz)^4\Gamma(1+\frac\xi z)}{\sin(\pi\frac{3h}{z})\Gamma(1+\frac{3h}{z})}\cdot\\
	&\quad \sum_{(d_1,d_2)\in \NN^2}q_1^{d_1}q_2^{d_2} \frac{(-1)^{d_2}\Gamma(d_1-d_2-\frac{\xi-h}{z})}{\Gamma(1+\frac hz+d_1)^4\Gamma(1+\frac \xi z+d_2)\Gamma(-\frac{3h}{z}-3d_1)z^{2d_2}}.\tag{3.3}
	\end{align*}
	Then define a sequence of functions $g_{d_2}(s,q_1)$ for each $d_2\in \NN$ as follows.
	\[g_{d_2}(s,q_1):=\frac{\Gamma(s-d_2-\frac{\xi-h}{z})q_1^{s}}{(e^{2\pi \sqrt{-1}s}-1)\Gamma(1+\frac hz+s)^4\Gamma(-\frac{3h}{z}-3s)}.\]
	It is clear that $g_{d_2}(s,q_1)$ is a meromorphic function in $s$ with simple poles at every integer, as well as $s=d_2+\frac{\xi-h}{z}-l$ for $l\in \NN$.
	
	We claim that the function $I^X(q_1,q_2)$ may be represented as the following integral:
	\begin{align*}
	I^X(q_1,q_2)&=3q_1^{h/z}q_2^{\xi/z}\frac{h\sin(\pi\frac{\xi-h}{z})\Gamma(1+\frac{\xi-h}z)\Gamma(1+\frac hz)^4\Gamma(1+\frac\xi z)}{\sin(\pi\frac{3h}{z})\Gamma(1+\frac{3h}{z})}\sum_{d_2\in \NN}\frac{(-q_2)^{d_2}}{\Gamma(1+\frac\xi z+d_2)z^{2d_2}} \int_{C^+} g_{d_2}(s,q_1)ds\tag{3.4},
	\end{align*} 
	where for a fixed $d_2$, the contour $C^+$ goes along the imaginary axis and closes to the right in such a way that only the simple poles at nonnegative integers are enclosed inside $C^+$.
	
	Indeed, according to the residue theorem, for each $d_2\in \bb N$ we have
	\begin{align*}
	\int_{C^+} g_{d_2}(s,q_1)ds&=2\pi \sqrt{-1}\sum_{d_1\in \bb N}\res_{s=d_1}g_{d_2}(s,q_1)\\
	&=\sum_{d_1\in \bb N}\frac{\Gamma(d_1-d_2-\frac{\xi-h}{z})q_1^{d_1}}{\Gamma(1+\frac hz+d_1)^4\Gamma(-\frac{3h}{z}-3d_1)}.\tag{3.5}
	\end{align*}
	Substuiting (3.5) into (3.4) for each $d_2\in \NN$, we obtain (3.3), hence the the claim follows. 
	
	To perform the analytic continuation, we notice that for $|q_1|$ sufficiently large, we may close up the imaginary axis to the left in such a way that all the remaining poles are enclosed in this contour, denoted by $C^-$. Using the residue theorem again, we obtain
	\begin{align*}
	\int_{C^-} g_{d_2}(s,q_1)ds&=2\pi \sqrt{-1}\sum_{l\in {\bb N}}\left(\res_{s=-l-1}g_{d_2}(s,q_1)+\res_{s=d_2+\frac{\xi-h}{z}-l}g_{d_2}(s,q_1)\right)\\
	&=2\pi \sqrt{-1}\sum_{l\in\bb N}\frac{\Gamma(-l-1-d_2-\frac{\xi-h}{z})q_1^{-l-1}}{\Gamma(\frac hz-l)^4\Gamma(-\frac{3h}{z}+3l+3)}+\\
	&\quad2\pi \sqrt{-1}\sum_{l\in {\bb N}}\frac{(-1)^lq_1^{d_2-l+\frac{\xi-h}{z}}}{(e^{2\pi \sqrt{-1}\frac{\xi-h}{z}}-1)\Gamma(1+l)\Gamma(1+\frac {\xi}z+d_2-l)^4\Gamma(3l-3d_2-\frac{3\xi}{z})}.\tag{3.6}
	\end{align*}
	Replacing $C^+$ by $C^-$ in (3.4), and substuiting (3.6) into (3.4) for each $d_2\in \NN$, we obtain the following analytic continuation of $I^X(q_1,q_2)$:
	\[\bar I^X(q_1,q_2)=3(q_1q_2)^{\xi/z}\sum_{(l,d_2)\in \NN^2}q_1^{d_2-l}(q_2)^{d_2}A_{l,d_2}+h^4f(q_1,q_2,\log q_1,\log q_2, h, \xi)\tag{3.7},\\\]
	for some $f(q_1,q_2,\log q_1,\log q_2, h, \xi)\in \CC[[q_1,q_2,\log q_1,\log q_2, h,\xi]]$ and
	\begin{align*}
	A_{l,d_2}&=\frac{(-1)^{l+d_2}(2\pi \sqrt{-1})h\sin\pi(\frac{\xi-h}{z})\Gamma(1+\frac{\xi-h}z)\Gamma(1+\frac hz)^4\Gamma(1+\frac\xi z)}{\sin\pi(\frac{3h}{z})\Gamma(1+\frac{3h}{z})(e^{2\pi \sqrt{-1}\frac{\xi-h}{z}}-1)\Gamma(1+\frac{\xi}{z}+d_2)\Gamma(1+l)\Gamma(1+\frac{\xi}z+d_2-l)^4\Gamma(3l-3d_2-\frac{3\xi} z)z^{2d_2}}\\
	&=\frac{(2\sqrt{-1})h\sin(\frac{3\xi}{z}\pi)\sin(\frac{\xi-h}{z}\pi)\Gamma(1+\frac{\xi-h}{z})\Gamma(1+\frac hz)^4\Gamma(1+\frac\xi z)}{\sin(\frac {3h}{z}\pi)(e^{2\pi \sqrt{-1}\frac{\xi-h}{z}}-1)\Gamma(1+\frac {3h}z)}\\
	&\quad \cdot \frac{\Gamma(1+3\frac\xi z+3d_2-3l)}{\Gamma(1+\frac{\xi}z+d_2)\Gamma(1+l)\Gamma(1+\frac\xi z+d_2-l)^4z^{2d_2}}.
	\end{align*}
	Replacing $q_1,q_2$ by $x,y$ using the change of variable, and repeatedly applying the relation $h^4=h\xi-\xi^2=0$, (3.7) eventually reduces to
	\[\bar I^X(x,y)=\frac{(3\xi)y^{\frac\xi z}\Gamma(1+\frac \xi z)^5}{\Gamma(1+\frac {3\xi}z)}\sum_{(i,j)\in\NN^2}x^{i}y^{j} \frac{\Gamma(1+3\frac\xi z+3j-3i)}{\Gamma(1+\frac{\xi}z+j)\Gamma(1+i)\Gamma(1+\frac\xi z+j-i)^4z^{2j}}.\tag{3.8}\]
	
	Comparing (3.8) with (3.2), it is clear that the following linear transformation does the job:
	\[L: \xi\mapsto p,\qquad \xi^2\mapsto p^2,\qquad \xi^3\mapsto p^3,\qquad \xi^4\mapsto p^4,\]
	hence we are done.
\end{proof}
Now we are in a position to state and prove our main theorem for the local model.
\begin{thm}\label{mytheorem}
	One may perform analytic continuation of $\cH(X)$ to obtain a $D$-module $\bar{\cH}(X)$, then there is a divisor $E$ in the extended K\"ahler moduli space and a submodule $\bar{\cH}^E(X)\subseteq \bar{\cH}(X)$ with maximum trivial $E$-monodromy such that
	$$\bar{\cH}^E(X)|_E\simeq \cH(Y),$$
	where $\bar{\cH}^E(X)|_E$ is the restriction to $E$.
\end{thm}
\begin{proof}
	We may identify the ambient part quantum $D$-module ${\cH}(X)$ with the local system attached to $I^X(x,y)$ (Remark 2.5). By \Cref{continue}, we see that $I^X(x,y)$ can be analytically continued to $\bar I^Y(x,y)$ up to a linear transformation. Let $\bar \cH(X)$ be the $D$-module that correpsonds to the system $\triangle_1'I=\triangle_2'I=0$.
	
	Next, we claim the components of $\bar I^Y(x,y)$ give a sub-solution space which has maximum trivial $x$-monodromy. Indeed, their $x$-monodromy is trivial by Lemma 3.4, and maximal because the remaining two solutions $I_5$ and $I_6$ have non-trivial $x$-monodromy due to their initial terms $x^{1/3}$ and $x^{2/3}$. Hence our claim follows. 
	
	Consider the sub-local system $\bar \cH^E(X)$ attached to the components of $\bar I^Y(x,y)$, since their $x$-monodromy is trivial, we may consider the natural restriction of $\bar I^Y(x,y)$ to $x=0$. By Lemma 3.4, we recover the $I$-function $I^Y(y)$ for $Y$ through this process, so the local system ${\cH}(Y)$ is obtained by restricting the local system of $\bar \cH^E(X)$ to $x=0$, which is viewed as a divisor $E$ in the extended K\"ahler moduli. Therefore the theorem is proved.
\end{proof}

\section{A global example}
In this section, we will study one particular example of cubic extremal transitions for Calabi-Yau 3-folds (which appears in \cite{02calabi}), and we will show that \Cref{myconj} holds in this case.

Let $Y$ be a quintic hypersurface in $\PP^4$ defined by the following homogenous equation:
\[x_0^2(tx_0^3+x_1^3+x_2^3+x_3^3+x_4^3)+x_0f(x_1,x_2,x_3,x_4)+g(x_1,x_2,x_3,x_4)=0,\quad(t\neq 0)\]
where $f$ and $g$ are generic homogenous polynomials of degree 4 and 5, respectively. $Y$ is generically a smooth Calabi-Yau 3-folds, which can be deformed into the following singular quintic $Y_0$ with singularity only at $[1:0:0:0:0]\in \PP^4$:
\[Y_0:\quad x_0^2(x_1^3+x_2^3+x_3^3+x_4^3)+x_0f(x_1,x_2,x_3,x_4)+g(x_1,x_2,x_3,x_4)=0.\]
Let $X$ be the Calabi-Yau 3-folds obtained by blowing up $Y_0$ at the triple point $[1:0:0:0:0]\in\PP^4$. Clearly $X$ is a smooth Calabi-Yau 3-fold and the exceptional divisor over $[1:0:0:0:0]\in\PP^4$ is the cubic surface defined by the following equation:
\[S: \{x_1^3+x_2^3+x_3^3+x_4^3=0\}\subseteq \PP^3.\]
The passage from $X$ to $Y$ is an example of global cubic extremal transitions. 
\subsection{The quantum $D$-modules}
On the same line of reasoing, $X$ is a hypersurface inside $\widetilde{X}=\PP(\cO_{\PP^3}\oplus\cO_{\PP^3}(-1))$. We will still adopt the same notations in Section 3, where
\begin{itemize}
	\item $\pi:\PP(\cO_{\PP^3}\oplus\cO_{\PP^3}(-1))\to \PP^3$ is the projection map, and $\cO_{\widetilde{X}}(1)$ is the anti-tautological line bundle over $\widetilde{X}=\PP(\cO_{\PP^3}\oplus\cO_{\PP^3}(-1))$.
	\item $h:=c_1(\pi^*\cO_{\PP^3}(1))$, $\xi:=c_1(\cO_{\widetilde{X}}(1))$. Let $q_1,q_2$ be the small parameters for $X$, correpsonding to $h$ and $\xi$.
	\item $p:=c_1(\cO_{\PP^4}(1))$. Let $y$ be the small parameter for $Y$, corresponding to $p$.
\end{itemize}

Then $X$ represents the divisor class $3h+2\xi$ in $\tilde{X}$, whereas $Y$ represents the divisor class $5p$ in $\hat{Y}=\PP^4$. The $I$-functions of $X$ and $Y$ are the following:
\[I^X(q_1,q_2): =(3h+2\xi)q_1^{h/z}q_2^{\xi/z}\sum_{(d_1,d_2)\in\NN^2} q_1^{d_1}q_2^{d_2}\frac{\prod^0\limits_{m=-\infty} (\xi-h+mz)\prod\limits_{m=1}^{3d_1+2d_2}(3h+2\xi+mz)}{\prod^{d_1}\limits_{m=1}(h+mz)^4\prod^{d_2}\limits_{m=1}(\xi+mz)\prod^{d_2-d_1}\limits_{m=-\infty}(\xi-h+mz)},\]
subject to the relation $h^4=h\xi-\xi^2=0$.
\[I^Y(y)=(5p)y^{p/z}\sum_{d\in \NN} \frac{\prod\limits_{m=1}^{5d}(5p+mz)}{\prod\limits_{m=1}^{d}(p+mz)^5},\]
subject to the relation $p^5=0$.

Following Section 3, we may also construct a hypergeometric series $I^{Y}(x,y)$ in $x$ and $y$ such that
\[\lim_{x\to 0}\bar{I}^Y(x,y)=I^Y(y),\]
where $\bar{I}^Y(x,y)$ is given by:
\[\bar I^Y(x,y):=(5p)e^{(p\log y)/z}\sum_{\substack{i\geqslant 0\\ j\geqslant 0}}x^iy^j\frac{\prod\limits_{m=-\infty}^{0}(p+mz)^4\prod\limits_{m=-\infty}^{5j-3i}(5p+mz)}{\prod\limits_{m=-\infty}^{j-i}(p+mz)^4\prod\limits_{m=1}^{j}(p+mz)\prod\limits_{m=1}^{i}(mz)\prod\limits_{m=-\infty}^{0}(5p+mz)},\,\]
subject to the relation $p^5=0$.

With the explicit formulas of $I^X(q_1,q_2)$, we have the following lemmas. (The differential operators below are also computed in \cite{MR3663796}.)

\begin{lemma}
	The components of  $I^X(q_1,q_2)$ give a full basis of solutions to the system $\triangle_1I=\triangle_2I=P_0I=0$ near the origin in $\CC^*$, where
	\[\triangle_1=(z\delta_{q_1})^4-q_1(3z\delta_{q_1}+2z\delta_{q_2}+z)(3z\delta_{q_1}+2z\delta_{q_2}+2z)(3z\delta_{q_1}+2z\delta_{q_2}+3z)(z\delta_{q_2}-z\delta_{q_1}),\]
	\[\triangle_2=(z\delta_{q_2})(z\delta_{q_2}-z\delta_{q_1})-q_2(3z\delta_{q_1}+2z\delta_{q_2}+z)(3z\delta_{q_1}+2z\delta_{q_2}+2z),\]
	\begin{align*}
	P_0=&-5(z\delta_{q_1})^3+2(z\delta_{q_1})^2(z\delta_{q_2})+15q_1(z\delta_{q_2}-z\delta_{q_1})(3z\delta_{q_1}+2z\delta_{q_2}+z)(3z\delta_{q_1}+2z\delta_{q_2}+2z)\\
	&\quad-4q_2(z\delta_{q_2})^2(3z\delta_{q_1}+2z\delta_{q_2}+z).
	\end{align*}
\end{lemma}

In a similar fashion, we introduce the following two hypergeometric series:
\[I_5(x,y)=x^{\frac13}\sum_{i\geqslant j\geqslant 0}x^iy^j\frac{(-1)^{i-j}\Gamma(\frac13+i-j)^4}{\Gamma(\frac 43+i)\Gamma(3i-5j+1)\Gamma(1+j)},\]
\[I_6(x,y)=x^{\frac 23}\sum_{i\geqslant j\geqslant 0}x^iy^j\frac{(-1)^{i-j}\Gamma(\frac23+i-j)^4}{\Gamma(\frac 53+i)\Gamma(3i-5j+2)\Gamma(1+j)}.\]

\begin{lemma}
	The components of $\bar I^Y(q_1,q_2)$, together with $I_5,I_6$, give a full basis to the system $\triangle_1'I=\triangle_2'I=P_0'I=0$, where
	\[\displaystyle \triangle_1'=x(z\delta_y-z\delta_x)^4-(5z\delta_y-3z\delta_x+3z)(5z\delta_y-3z\delta_x+2z)(5z\delta_y-3z\delta_x+z+z)(z\delta_x),\]
	\[\displaystyle \triangle_2'=(z\delta_{x})(z\delta_{y})-xy(5\delta_y-3\delta_x+z)(5\delta_y-3\delta_x+2z),\]
	\begin{align*}
    P_0'=&-5(z\delta_{y}-z\delta_{x})^3+2(z\delta_{y}-z\delta_{x})^2(z\delta_{y})+15x^{-1}(z\delta_{x})(5z\delta_{y}-3z\delta_{x}+z)(5z\delta_{y}-3z\delta_{x}+2z)\\
	&\quad-4xy(z\delta_{y}-z\delta_{x})^2(5z\delta_{y}-3z\delta_{x}+z).
	\end{align*}
\end{lemma}

The differential equation systems $\{\triangle_1I=\triangle_2I=P_0I=0\}$ and $\{\triangle'_1I=\triangle'_2I=P_0'I=0\}$ both have 6-dimensional solution spaces. We can relate them using an appropriate change of variable as follows.
\begin{lemma}
	The change of variable $x\mapsto q_1^{-1}$ and $y\mapsto q_1q_2$ induces an equivalence between the differential equation systems $\{\triangle_1I=\triangle_2I=P_0I=0\}$ and $\{\triangle'_1I=\triangle'_2I=P_0'I=0\}$.
\end{lemma}
\begin{proof}
Notice that under the change of variable $x\mapsto q_1^{-1}$ and  $y\mapsto q_1q_2$, the differential operators have the following relations: 
\[\delta_{q_1}=\delta_{y}-\delta_x,\quad \delta_{q_2}=\delta_y.\]
It is straightfoward to check that this change of variable converts $\triangle_1I=\triangle_2I=P_0I=0$ directly into $\triangle_1'I=\triangle_2'I=P_0I=0$, provided that $x\neq 0$ Hence we are done. 
\end{proof}
%Hence the $D$-modules attached to the two Picard-Fuchs ideals can be analytic continue to each other, with the change of coordinates given by $x\to q_1^{-1}$ and $y\mapsto q_1q_2$. Note also that the components of $\bar I^Y(x,y)$ has trivial $x$-monodromy. By a similar argument as in Theorem 3.1, we conclude that the \Cref{myconj} holds in this case.
\subsection{Analytic continuation of the $I$-function}
Since $I^X(q_1,q_2)$ and $\bar I^Y(x,y)$ satisfy the the same differential equation up to a change of variable, we would expect that $I^X(q_1,q_2)$ can be analytic continued to $\bar I^Y(x,y)$ up to a linear transformation. Indeed, we may still work out this analytic continuation explicitly using Mellin-Barnes method, which is parallel to Section 3.2.

First, we rewrite $I^X(q_1,q_2)$ and $\bar I^Y(x,y)$ using Lemma 3.6 as follows.
\begin{align*}
I^X(q_1,q_2)&=(3h+2\xi)q_1^{\frac hz}q_2^{\frac \xi z}\cdot\frac{\Gamma(1+\frac{\xi-h}z)\Gamma(1+\frac hz)^4\Gamma(1+\frac\xi z)}{\Gamma(1+\frac{3h+2\xi}{z})}\cdot \\
&\quad\sum_{\substack{d_1\geqslant 0\\ d_2\geqslant 0}}q_1^{d_1}q_2^{d_2} \frac{\Gamma(1+\frac {3h+2\xi}{z}+3d_1+2d_2)}{\Gamma(1+\frac hz+d_1)^4\Gamma(1+\frac \xi z+d_2)\Gamma(1+\frac{\xi-h}{z}+(d_2-d_1))},\tag{4.1}
\end{align*}
subject to the relation $h^4=h\xi-\xi^2=0.$
\begin{equation*}
\bar I^Y(x,y)=\frac{(5p)y^{\frac pz}\Gamma(1+\frac pz)^5}{\Gamma(1+\frac {5p}{z})}\sum_{\substack{i\geqslant 0\\ j\geqslant 0}}x^iy^j \frac{\Gamma(1+\frac{5p}{z}+5j-3i)}{\Gamma(1+\frac pz+j-i)^4\Gamma(1+\frac pz+j)\Gamma(1+i)},\tag{4.2}
\end{equation*}
subject to relation $p^5=0.$

For every $d_2\in \NN$, we define the following function 
\begin{align*}
\varphi_{d_2}(s):&=(-1)^{d_2}\displaystyle\frac{\sin(\frac{\xi-h}z)\pi}{\sin(\frac{3h+2\xi}z)\pi}\cdot \frac{\sin(-\frac{3h+2\xi}{z}-3s-2d_2)\pi}{\sin(s-d_2-\frac{\xi-h}{z})\pi}\\
&=(-1)^{d_2}\frac{\sin(\frac{\xi-h}z)\pi}{\sin(\frac{3h+2\xi}z)\pi}\cdot \frac{\pi/\sin(s-d_2-\frac{\xi-h}{z})\pi}{\pi/\sin(-\frac{3h+2\xi}{z}-3s-2d_2)\pi}\\
&=(-1)^{d_2}\frac{\sin(\frac{\xi-h}z)\pi}{\sin(\frac{3h+2\xi}z)\pi}\cdot \frac{\Gamma(1+\frac{\xi-h}z+d_2-s)\Gamma(s-d_2-\frac{\xi-h}{z})}{\Gamma(1+\frac{3h}z+3s+2d_2)\Gamma(-\frac{3h+2\xi}z-3s-2d_2)}.
\end{align*}
By definition, $\varphi_{d_2}(s)$ is periodic with period 1, and takes value 1 at every integer $s\in \bb Z$. It follows that $I^X(x,y)$ can be further written as
\begin{align*}
I^X(q_1,q_2)&=(3h+2\xi)q_1^{h/z}q_2^{\xi/z}\frac{\Gamma(1+\frac{\xi-h}z)\Gamma(1+\frac hz)^4\Gamma(1+\frac\xi z)}{\Gamma(1+\frac{3h+2\xi}{z})}\cdot\\
&\quad \sum_{(d_1,d_2)\in \NN^2}q_1^{d_1}q_2^{d_2} \frac{\Gamma(1+\frac {3h+2\xi}{z}+3d_1+2d_2)\varphi_{d_2}(d_1)}{\Gamma(1+\frac hz+d_1)^4\Gamma(1+\frac \xi z+d_2)\Gamma(1+\frac{\xi-h}{z}+(d_2-d_1))},\\
&=(3h+2\xi)q_1^{h/z}q_2^{\xi/z}\frac{\sin(\pi\frac{\xi-h}{z})\Gamma(1+\frac{\xi-h}z)\Gamma(1+\frac hz)^4\Gamma(1+\frac\xi z)}{\sin(\pi\frac{3h+2\xi}{z})\Gamma(1+\frac{3h+2\xi}{z})}\cdot\\
&\quad \sum_{(d_1,d_2)\in \NN^2}q_1^{d_1}q_2^{d_2} \frac{(-1)^{d_2}\Gamma(d_1-d_2-\frac{\xi-h}{z})}{\Gamma(1+\frac hz+d_1)^4\Gamma(1+\frac \xi z+d_2)\Gamma(-\frac{3h+2\xi}{z}-3d_1-2d_2)}.\tag{4.3}
\end{align*}
Similarly, we define a sequence of functions $g_{d_2}(s,q_1)$ for each $d_2\in \NN$ as follows.
\[g_{d_2}(s,q_1):=\frac{\Gamma(s-d_2-\frac{\xi-h}{z})q_1^{s}}{(e^{2\pi \sqrt{-1}s}-1)\Gamma(1+\frac hz+s)^4\Gamma(-\frac{3h+2\xi}{z}-3s-2d_2)}.\]
We see that $g_{d_2}(s,q_1)$ is a meromorphic function in $s$ with simple poles at every integer, as well as $s=d_2+\frac{\xi-h}{z}-l$ for $l\in \NN$. We claim that $I^X(x,y)$ admits the following integral representation.
\begin{align*}
I^X(q_1,q_2)&=(3h+2\xi)q_1^{h/z}q_2^{\xi/z}\frac{\sin(\pi\frac{\xi-h}{z})\Gamma(1+\frac{\xi-h}z)\Gamma(1+\frac hz)^4\Gamma(1+\frac\xi z)}{\sin(\pi\frac{3h+2\xi}{z})\Gamma(1+\frac{3h+2\xi}{z})}\\
&\quad \cdot \sum_{d_2\in \NN}\frac{(-q_2)^{d_2}}{\Gamma(1+\frac\xi z+d_2)} \int_{C^+} g_{d_2}(s,q_1)ds\tag{4.4},
\end{align*} 
where for a fixed $d_2$, the contour $C^+$ goes along the imaginary axis and closes to the right in such a way that only the simple poles at nonnegative integers are enclosed inside $C^+$.

Again, this follows from the residue computation. For each $d_2\in \bb N$ we have
\begin{align*}
\int_{C^+} g_{d_2}(s,q_1)ds&=2\pi \sqrt{-1}\sum_{d_1\in \bb N}\res_{s=d_1}g_{d_2}(s,q_1)\\
&=\sum_{d_1\in \bb N}\frac{\Gamma(d_1-d_2-\frac{\xi-h}{z})q_1^{d_1}}{\Gamma(1+\frac hz+d_1)^4\Gamma(-\frac{3h+2\xi}{z}-3d_1-2d_2)}.\tag{4.5}
\end{align*}
Substuiting (4.5) into (4.4) gives (4.3). Hence our claim is proved.

To perform the analytic continuation, we notice that for $|q_1|$ sufficiently large, we may close up the imaginary axis to the left in such a way that all the remaining poles are enclosed in this contour, denoted by $C^-$. Using the residue theorem again, we obtain
\begin{align*}
\int_{C^-} g_{d_2}(s,q_1)ds&=2\pi \sqrt{-1}\sum_{l\in {\bb N}}\left(\res_{s=-l-1}g_{d_2}(s,q_1)+\res_{s=d_2+\frac{\xi-h}{z}-l}g_{d_2}(s,q_1)\right)\\
&=2\pi \sqrt{-1}\sum_{l\in\bb N}\frac{\Gamma(-l-1-d_2-\frac{\xi-h}{z})q_1^{-l-1}}{\Gamma(\frac hz-l)^4\Gamma(-\frac{3h+2\xi}{z}+3l-2d_2+3)}+\\
&\quad2\pi \sqrt{-1}\sum_{l\in {\bb N}}\frac{(-1)^lq_1^{d_2-l+\frac{\xi-h}{z}}}{(e^{2\pi \sqrt{-1}\frac{\xi-h}{z}}-1)\Gamma(1+l)\Gamma(1+\frac {\xi}z+d_2-l)^4\Gamma(3l-5d_2-\frac{5\xi}{z})}.\tag{4.6}
\end{align*}
Replacing $C^+$ by $C^-$ in (4.4), and substuiting (4.6) into (4.4) for each $d_2\in \NN$, we obtain the following analytic continuation of $I^X(q_1,q_2)$.
\[\bar I^X(q_1,q_2)=(3h+2\xi)(q_1q_2)^{\xi/z}\sum_{(l,d_2)\in \NN^2}q_1^{-l}(q_1q_2)^{d_2}A_{l,d_2}+h^4f(q_1,q_2,\log q_1,\log q_2, h, \xi)\tag{4.7},\\\]
for some $f(q_1,q_2,\log q_1,\log q_2, h, \xi)\in \CC[[q_1,q_2,\log q_1,\log q_2, h,\xi]]$ and
\begin{align*}
A_{l,d_2}&=\frac{(2\sqrt{-1})h\sin(\frac{5\xi}{z}\pi)\sin(\frac{\xi-h}{z}\pi)\Gamma(1+\frac{\xi-h}{z})\Gamma(1+\frac hz)^4\Gamma(1+\frac\xi z)}{\sin(\frac {3h+2\xi}{z}\pi)(e^{2\pi \sqrt{-1}\frac{\xi-h}{z}}-1)
\Gamma(1+\frac {3h+2\xi}z)}\\
&\quad\cdot \frac{\Gamma(1+5\frac\xi z+5d_2-3l)}{\Gamma(1+\frac{\xi}z+d_2)\Gamma(1+l)\Gamma(1+\frac\xi z+d_2-l)^4}.
\end{align*}
Replacing $q_1,q_2$ by $x,y$ using the change of variable, and repeatedly applying the relation $h^4=h\xi-\xi^2=0$, (4.7) eventually reduces to
\[\bar I^X(x,y)=\frac{(5\xi)y^{\frac\xi z}\Gamma(1+\frac \xi z)^5}{\Gamma(1+\frac {3\xi}z)}\sum_{(i,j)\in\NN^2}x^{i}y^{j} \frac{\Gamma(1+5\frac\xi z+5j-3i)}{\Gamma(1+\frac{\xi}z+j)\Gamma(1+i)\Gamma(1+\frac\xi z+j-i)^4}.\]
Thus using the linear transformation $L:H^*(\widetilde X)\to H^*(\widetilde Y)$ given by
\[L: \xi\mapsto p,\qquad \xi^2\mapsto p^2,\qquad \xi^3\mapsto p^3,\qquad \xi^4\mapsto p^4,\]
we have
\[\bar{I}^Y(x,y)=L\circ \bar{I}^X(x,y),\]
and 
\[I^Y(y)=\lim_{x\to 0}\bar{I}^Y(x,y)=\lim_{x\to 0}L\circ \bar{I}^X(x,y).\]
Following a similar argument as in \Cref{mytheorem}, now we may conclude that \Cref{myconj} holds in our case, that is
\begin{thm}
	For the cubic extremal transition between two Calabi-Yau 3-folds $X$ and $Y$ given above, one may perform analytic continuation of $\cH(X)$ to obtain a $D$-module $\bar{\cH}(X)$, then there is a divisor $E$ in the extended K\"ahler moduli space and a submodule $\bar{\cH}^E(X)\subseteq \bar{\cH}(X)$ with maximum trivial $E$-monodromy such that
	$$\bar{\cH}^E(X)|_E\simeq \cH(Y).$$
\end{thm}

\section{Connection to FJRW $D$-modules}
In this section, we will explore the phenomenon of rank reduction in cubic extremal transitions, and provide a partial explanation for it, namely, the FJRW theory of the cubic singularity. 

In the local model or the global example of cubic extremal transitions studied in Section 3 and Section 4, we notice that the ambient quantum $D$-module of $\cH(X)$ has rank 6, whereas the ambient quantum $D$-module $\cH(Y)$ has rank 4. In general, one would expect the Gromov-Witten theory of $X$ recovers the Gromov-Witten theory of $Y$. This is partly due to the nature of extremal transitions: as noted in \cite{MR2240431},  extremal transitions would generally decrease the K\"ahler moduli. Since $Y$ is obtained by deforming the cubic singularity, we expect the cubic singularity would account for this rank reduction.

The quantum theory of the Landau-Ginzburg A-model was worked out by Fan, Jarvis, and Ruan in a series papers \cite{MR3112508,MR3043578,MR2400605}, based on Witten's proposal \cite{MR1215968}, and now this theory is commonly known as FJRW theory. 
According to the LG/CY correspondence \cite{MR2672282}, roughly speaking, the FJRW theory of a pair $(W,G)$ satisfying the so-called Calabi-Yau condition is equivalent to the Gromov-Witten theory of the hypersurface defined by the equation $W=0$ in a weighted projective space. The basic strategy is to relate the $I$-functions associated  to the two theories by analytic continuation. However, as explained in P. Acosta's work \cite{2014arXiv1411.4162A}, when the Calabi-Yau condition for $(W,G)$ fails, the $I$-function on one side has an irregular singularity at infinity and on the other side becomes a formal function with zero radius of convergence. In his work, an analogous correspondence for Fano/General Type varieties is developed by appealing to the theory of asymptotic expansion. It is shown that in the Fano case, the Gromov-Witten $I$-function can be analytically continued,  up to a linear transformation, to obtain the so-called \emph{regularized} FJRW $I$-function, which recovers the ordinary FJRW $I$-function via asymptotic expansion. We adopt the notions of \cite{2014arXiv1411.4162A} in our following settings.

Let $W=x^3+y^3+z^3+w^3$, which is precisely the local equation for the singularity in cubic extremal transitions. We consider the FJRW $I$-function associated to the pair $(W,G)$, where $G=\langle J_W\rangle$ is generated by 
$$J_W:=\left( \exp\left(\frac{2\sqrt{-1}\pi}{3}\right),\exp\left(\frac{2\sqrt{-1}\pi}{3}\right),\exp\left(\frac{2\sqrt{-1}\pi}{3}\right),\exp\left(\frac{2\sqrt{-1}\pi}{3}\right)\right).$$
The (small) FJRW $I$-function of $(W,G)$ is a formal power series given by:
\[I_{FJRW}(t,z=1):=-\sum_{l=0}^\infty t^{3l+1}\frac{(-1)^l\Gamma(l+\frac13)^4}{(3l)!\Gamma(\frac13)^4}\phi_0+\sum_{l=0}^\infty t^{3l+2}\frac{(-1)^l\Gamma(l+\frac23)^4}{(3l)!\Gamma(\frac23)^4}\phi_1,\tag{5.1}\]
where $\phi_0$ and $\phi_1$ are generators of the narrow sector of the state space of $(W,G)$. 

It is direct to see that the series (5.1) has zero radius of convergence. There is a natural way to regularize this function, as introduced in \cite{2014arXiv1411.4162A}.
\begin{definition}
	The regularized FJRW $I$-function is defined in \cite{2014arXiv1411.4162A} as follows.
	$$I_{FJRW}^{reg}(\tau):=\sum_{l=0}^\infty \frac{\tau^{l+\frac13}(-1)^{3l+1}\Gamma(l+\frac13)^4}{(3l)!\Gamma(l+\frac43)\Gamma(\frac13)^4}\phi_0+\sum_{l=0}^\infty \frac{\tau^{l+\frac23}(-1)^{3l+2}\Gamma(l+\frac23)^4}{(3l+1)!\Gamma(l+\frac53)\Gamma(\frac23)^4}\phi_1,$$
	and we define the \emph{FJRW $D$-module} of $(W,G)$ as the $D$-module (or the local system) attached to the components of the regularized FJRW $I$-function as above, denoted by ${\cL}(W)$
\end{definition}
\begin{remark} 
	The components of the regularized FJRW $I$-function of $(W,G)$ are both analytic near $\tau=0$. It is shown in \cite{2014arXiv1411.4162A} that this $I_{FJRW}^{reg}(\tau)$ may be used to recover the ordinary FJRW $I$-function $I_{FJRW}(t,z=1)$, through the method of asymptotic expansion. So $I_{FJRW}^{reg}(\tau)$ can be viewed as encoding the genus zero data of the (narrow part) FJRW theory of $(W,G)$.
\end{remark}

In Section 3 (or Section 4), we defined the hypergeometric series $I_5(x,y)$ and $I_6(x,y)$ as the "extra" solutions coming from the analytic continuation of the ambient part quantum $D$-module of $X$. Interestingly, these functions are directly related to the regularized FJRW theory of $(W,G)$ in the following way:
\begin{prop}
	The function $I_5(x,y)$ (resp. $I_6(x,y)$) is analytic in $y$ near the origin, and also it has trivial monodromy around $y=0$. It recovers, up to scalar multiple, the coefficient of $\phi_0$ (resp. $\phi_1$) in $I_{FJRW}^{reg}(\tau)$ by setting $x=\tau$, $z=1$ and $y\to 0$.
\end{prop}
\begin{proof} It is analytic by a ratio test, and monodromy around $y=0$ is trivial because it is an ordinary power series in $y$ for every fixed value of $x$. As for the last part, it is straightforward to check. 
\end{proof}

As a direct corollary, we have
\begin{thm}
	Let $\bar \cH(X)$ be the $D$-module obtained by analytic continuation considered in Theorem 3.8 (or Theorem 4.4), then there is another divisor $F$ in the extended K\"ahler moduli and a submodule  $\bar{\cH}^F(X)\subseteq \bar{\cH}(X)$ which has maximal trivial $F$-monodromy such that
	$$\bar{\cH}^F(X)/\bar{\cH}^E(X)|_F\simeq {\cL}(W),$$
	where ${\cL}(W)$ represents the FJRW $D$-module of the pair $(W,G)$, in which $W=x^3+y^3+z^3+u^3$ and $G=\langle J_W\rangle$.
\end{thm}
\begin{proof}
	Using the notations from Section 3 (or Section 4), we let $I_1,I_2,I_3,I_4$ be the components of $\bar I^Y(x,y)$, where $I_1$ represents the constant term in $\bar I^Y(x,y)$. We notice that $I_1$ is analytic in an open neighborhood of the origin, and $I_5$ and $I_6$ have trivial monodromy around $y=0$. Thus to single out the components $I_5,I_6$, we should take the maximal components that have trivial $y$-monodromy, namely $I_1,I_5,I_6$, and then modulo the maximal components with trivial $x-$mondromy, namely $I_1,I_2,I_3,I_4$. In terms of $D$-modules, we just need to take the restriction $\bar{\cH}^F(X)/\bar{\cH}^E(X)|_F$, in which the divisor $F$ corresponds to $y= 0$ in the extended K\"ahler moduli. Hence the theorem is proved.
\end{proof}

\bibliography{research}
\bibliographystyle{plain}

\end{document}